\documentclass[12pt]{article}
\usepackage{oldlfont}
\usepackage{enumerate}
\usepackage{amsmath}
\usepackage{amssymb}
\usepackage{amsthm}
\usepackage{mathrsfs}
\usepackage{amscd}
\usepackage{exscale}
\usepackage{latexsym}
\usepackage[all]{xy}
\usepackage{hyperref}
\usepackage{fancyhdr}
\usepackage{layout}
\usepackage{color}
\usepackage{graphicx}

\DeclareMathAlphabet{\mathpzc}{OT1}{pzc}{m}{it}

\begin{document}

\baselineskip=17pt

\pagestyle{headings}

\numberwithin{equation}{section}

\makeatletter                                                           

\def\section{\@startsection {section}{1}{\z@}{-5.5ex plus -.5ex         
minus -.2ex}{1ex plus .2ex}{\large \bf}}                                 


\pagestyle{fancy}
\renewcommand{\sectionmark}[1]{\markboth{ #1}{ #1}}
\renewcommand{\subsectionmark}[1]{\markright{ #1}}
\fancyhf{} 
\fancyhead[LE,RO]{\slshape\thepage}
\fancyhead[LO]{\slshape\rightmark}
\fancyhead[RE]{\slshape\leftmark}

\addtolength{\headheight}{0.5pt} 
\renewcommand{\headrulewidth}{0pt} 

\newtheorem{thm}{Theorem}
\newtheorem{mainthm}[thm]{Main Theorem}
\newtheorem*{T}{Theorem 1'}

\newcommand{\ZZ}{{\mathbb Z}}
\newcommand{\GG}{{\mathbb G}}
\newcommand{\Z}{{\mathbb Z}}
\newcommand{\RR}{{\mathbb R}}
\newcommand{\NN}{{\mathbb N}}
\newcommand{\GF}{{\rm GF}}
\newcommand{\QQ}{{\mathbb Q}}
\newcommand{\CC}{{\mathbb C}}
\newcommand{\FF}{{\mathbb F}}

\newtheorem{lem}[thm]{Lemma}
\newtheorem{cor}[thm]{Corollary}
\newtheorem{pro}[thm]{Proposition}
\newtheorem*{proposi}{Proposition \ref{pro:pro63}}

\newtheorem{proprieta}[thm]{Property}
\newcommand{\pf}{\noindent \textbf{Proof.} \ }
\newcommand{\eop}{$_{\Box}$  \relax}
\newtheorem{num}{equation}{}

\theoremstyle{definition}
\newtheorem{rem}[thm]{Remark}
\newtheorem*{D}{Definition}

\newtheorem{Def}{Definition}

\newcommand{\nsplit}{\cdot}
\newcommand{\G}{{\mathfrak g}}
\newcommand{\GL}{{\rm GL}}
\newcommand{\SL}{{\rm SL}}
\newcommand{\SP}{{\rm Sp}}
\newcommand{\LL}{{\rm L}}
\newcommand{\Ker}{{\rm Ker}}
\newcommand{\la}{\langle}
\newcommand{\ra}{\rangle}
\newcommand{\PSp}{{\rm PSp}}
\newcommand{\U}{{\rm U}}
\newcommand{\GU}{{\rm GU}}
\newcommand{\Aut}{{\rm Aut}}
\newcommand{\Alt}{{\rm Alt}}
\newcommand{\Sym}{{\rm Sym}}

\newcommand{\isom}{{\cong}}
\newcommand{\z}{{\zeta}}
\newcommand{\Gal}{{\rm Gal}}

\newcommand{\F}{{\mathbb F}}
\renewcommand{\O}{{\cal O}}
\newcommand{\Q}{{\mathbb Q}}
\newcommand{\R}{{\mathbb R}}
\newcommand{\N}{{\mathbb N}}
\newcommand{\E}{{\mathcal{E}}}



\vskip 0.5cm

\title{On local-global divisibility by $p^n$\\ in elliptic curves}
\author{Laura Paladino, Gabriele Ranieri, Evelina Viada\footnote{Supported by the Swiss National Science Foundation SNF}}
\date{  }
\maketitle

\vskip 1.5cm

\begin{abstract}
Let $p$ be a prime number and let $k$ be a number field, which does not contain the field $\Q ( \zeta_p + \overline{\zeta_p} )$. Let $\mathcal{E}$ be an elliptic curve defined over $k$. We prove that if there are no $k$-rational torsion points of exact order $p$ on $\E$, then the local-global  principle holds for divisibility by  $p^n$, with $n$ a natural number.
As a consequence of the deep theorem of Merel, for $p$ larger than a constant depending only on the degree of $k$, there are no counterexamples to the local-global divisibility principle. Nice and deep works give explicit small constants for elliptic curves defined over a number field of degree at most $5$ over $\QQ$.

\end{abstract}

\section{Introduction}
Let $k$ be a number field  and let ${\mathcal{A}}$ be a commutative algebraic group defined over $ k $.
Several papers have been written on  the following classical question, known as \emph{Local-Global Divisibility Problem}.

\par\bigskip\noindent  P{\small ROBLEM}: \emph{Let $P\in {\mathcal{A}}(k)$. Assume that for all but finitely many valuations $v\in k$, there exists $D_v\in {\mathcal{A}}(k_v)$ such that $P=qD_v$, where $q$ is a positive integer. Is it possible to conclude that there exists $D\in {\mathcal{A}}(k)$ such that $P=qD$?}

\par\bigskip\noindent  By  B\'{e}zout's identity, to get answers for a general integer it is sufficient to solve it for powers $p^n$ of a prime. In the classical case of ${\mathcal{A}}={\mathbb{G}}_m$, the answer is positive for $ p $ odd, and negative for instance for $q=8$ (and $P=16$) (see for example \cite{AT}, \cite{Tro}).

\bigskip  For general commutative algebraic groups,  R. Dvornicich and U. Zannier  gave some general cohomological criteria, sufficient to answer the question (see \cite{DZ} and \cite{DZ3}).
 Using these criteria, they found a number of examples an counterexamples to the local-global principle when $\mathcal{A}$ is an elliptic curve or a torus. Further examples in a torus are given by M. Illengo \cite{Ill}.
For an elliptic curve $\E$, the local-global principle holds for divisibility by any prime $ p $. Furthermore, they provide a geometric criterium: if $ \E $ does not admit any $k$-isogeny of degree $ p $, then the local-global principle holds for divisibility by  $p^n$.
Theorems of Serre and of Mazur (see \cite{Ser} and \cite{Maz}) prove that such an isogeny exists only for $p\le c(k,\E)$,  where $c ( k, \E )$ is a constant depending on $ k $ and $\E$, and on elliptic curves over $\Q$, for $p\in S_1=\{2,3,5,7,11,13,17,19,37,43,67,163\}$.
 Thus the local-global principle holds in general for $p^n$ with  $p>c(k,\E)$ and in elliptic curves over $\QQ$ it suffices $p\notin S_1$.\\

The present work answers to a question of Dvornicich and Zannier:  \emph{Can one make the constant depending only
on the field $k$ and not on $\mathcal{E}$? } In a first paper \cite{P-R-V}, we give positive answer for the very special case of divisibility by $p^2$. Here we essentially give a strong geometric criterium: if $\E$ does not admit any $k$-rational torsion point of exact order  $ p $, then the local-global principle holds for divisibility by $p^n$. In view of the  deep Merel  theorem we give a general positive answer to the above question. More precisely the constant depends only on the degree of $k$. In an unpublished work, Oesterl\'e \cite{Oes} showed that there is no $k$-torsion of exact order larger than  $( 3^{[k:\Q] /2} + 1 )^2$.
This constant is not sharp and the bound is expected to be polynomial in the degree. Sharp bounds are hard. They are known only for fields of small degree.The effective Mazur Theorem \cite{Maz2} for elliptic curves over $\QQ$ allows us to shrunk the set $S_1$ to $\widetilde{S}_1=\{2,3,5,7\}$. The results of Kamienny~\cite{Ke-Mo}, Kenku and Momose \cite{Ke-Mo2} and works of Parent \cite{Par} and \cite{Par2} provide the potential minimal sets $S_2=S_3=\{2,3,5,7,11,13\}$ for elliptic curves over quadratic and cubic fields. Recent unpublished works by  Kamienny, Stein and Stoll \cite{K-S-S} and Derickx, Kamienny, Stein and Stoll \cite{D-K-S-S} give the potential minimal sets $S_4=\{2,3,5,7,11,13,17\}$ and $S_5=\{2,3,5,7,11,13,17,19 \}$ for elliptic curves over fields of degree $4$, respectively $5$, over $\QQ$.  Then, outside these sets the local-global-divisibility principle holds. The minimality of such sets for the local-global problem remains an open question, as only  counterexamples for $2^n$ and  $3^n$, for all $n\geq 2$ are
 known (\cite{DZ2}, \cite{Pal}, \cite{Pal2} and \cite{Pal3}).
 \begin{thm}\label{main}
Let $ p $ be a prime number and let $ n $ be a positive integer.
Let $\E$ be an elliptic curve defined over a number field $ k $, which does not contain the field  $\Q ( \zeta_p + \overline{\zeta_p} )$.
Suppose that $\E$ does not admit any $k$-rational torsion point of exact order $ p $.
Then, a point $P \in \E ( k )$  is locally divisible by $p^n$ in $\E ( k_v )$ for all but finitely many valuations $ v $ if and only if $ P $ is globally divisible by $p^n$ in $\E ( k )$.
\end{thm}
The mentioned cohomological criterium of Dvornicich and Zannier asserts that if the local cohomology of $G_n:=\textrm{Gal}(k({\mathcal{E}}[p^n])/k)$  is trivial, then there are no counterexamples (see Definition \ref{defloc} and Theorem \ref{si1}). Under the hypotheses of our theorem, we prove that the local cohomology of $ G_n $ is trivial. For divisibility by $p^2$, it happens that the structure of $G_2$ and consequently of its local cohomology is quite simple. The structure of $ G_n $ is however quite intricate. Thus, for the general case we cannot apply a direct approach like in the simpler case of the divisibility by $p^2$. Using an induction, we show that the groups $ G_n $ are generated by diagonal, strictly lower triangular and strictly upper triangular matrices. In addition we detect a special diagonal element in $G_n$. If there are no $k$-torsion points of exact order $ p $, the local cohomology of these subgroups or of their commutators is trivial. Thanks to the special diagonal element, we glue together the cohomologies and we conclude that the local cohomology of $ G_n $ is trivial, too.

As a nice consequence  of the deep theorem of L. Merel we produce a complete positive answer  to the question of Dvornicich and Zannier.

\begin{cor}\label{cor:cor42}
Let $ \E $ be an elliptic curve defined over any number field $k$.
Then, there exists a constant ${C} ( [k:\QQ] )$, depending only on the degree of $ k $, such that the local-global principle holds for divisibility by any power $p^n$ of primes $p > {C} ( [k:\QQ] )$. In addition $C( [k:\Q] ) \leq ( 3^{[k:\Q] /2} + 1 )^2$.
\end{cor}

\begin{proof} By \cite{Mer}, for every number field $ k $, there exists a constant $C_{merel} ( [k:\QQ] )$ depending only on the degree of $k$, such that, for every prime $p > C_{merel}  ( [k:\QQ] )$, no elliptic curve defined over $ k $ has a $k$-rational torsion point of exact order $p$.

\noindent Let $p_0$ be the largest prime such that $ k $ contains the field $\Q ( \zeta_{p_0} + \overline{\zeta_{p_0}} )$. Observe that $p_0 \leq 2 [k: \Q] + 1$.
Set
\[
{C} ( [k:\QQ] ) = {\rm max} \{ p_0, C_{merel} ( [k:\QQ] ) \}.
\] In an unpublished work, Oesterl\'e \cite{Oes} showed that $C_{merel} ( [k:\Q] ) \leq ( 3^{[k:\Q] /2} + 1 )^2$.
\noindent Then, apply Theorem~\ref{main}.
\end{proof}

The famous Mazur's Theorem
and further explicit versions of Merel's Theorem give:

  \begin{cor} Let \begin{description}
  \item[ ] ${C} ( 1 )=7$ for $\Q$;
  \item[ ] $C ( 2 )=13$ for quadratic fields;
  \item[ ] $C ( 3 )=13$ for cubic fields;
  \item[ ] $C ( 4 )=17$ for fields of
  degree $4$ over $\Q$;
  \item[ ] $C ( 5 )=19$ for fields of degree $5$ over $\Q$.

\end{description}
Let $\E$ be an elliptic curve defined over any number field of degree $d=1,2,3,4,5$. Then the local-global principle holds for divisibility by any power $p^n$ of primes $p > {C} ( d )$.
\end{cor}
\begin{proof}
Let $ k $ be a field of degree $d$ over $\Q$.
Observe that, for every $p > C ( d )$, $k$ does not contain $\Q ( \zeta_p + \overline{\zeta_p} )$.
Then, it suffices to replace the known  explicit constant for the previous corollary.
By the famous Mazur's Theorem (see \cite{Maz2}), no elliptic curve defined over $\Q$ has a rational point of exact prime order larger than $7$, then $C(1)=7$. By Kamienny ~\cite{Ke-Mo}, Kenku and Momose \cite{Ke-Mo2}, Parent \cite{Par} and \cite{Par2}, no elliptic curve defined over a quadratic or a cubic number field $ k $ has a $k$-rational point of exact prime order larger than $13$. Then $C(2)=C(3)=13$.  Further recent works in progress  by  Kamienny, Stein and Stoll \cite{K-S-S} and Derickx, Kamienny, Stein and Stoll \cite{D-K-S-S} exclude $k$-rational point of exact prime order larger than $17$, respectively $19$, for elliptic curves over number fields of degree $4$, respectively $5$. So $C(4)=17$ and $C(5)=19$. \end{proof}

\bigskip\noindent \emph{Acknowledgments}.  We deeply thank the referee for nice comments and for pointing out the works of Parent, Derickx, Kamienny, Stein and Stoll which give further explicit applications of our theorem. We would like to kindly thank F. Gillibert for her interesting remarks. The third author thanks the Swiss National Science Foundation for financial support.

\section{Preliminary results}\label{sec2}

Let $k$ be a number field and let $\mathcal{E}$ be an elliptic curve defined over $k$.
Let $p$ be a prime. 
For every positive integer $n$, we denote by ${\mathcal{E}}[p^n]$  the $p^n$-torsion subgroup of ${\mathcal{E}}$ and by $K_n = k({\mathcal{E}}[p^n])$  the number field obtained by adding to $k$ the coordinates of the $p^n$-torsion points of ${\mathcal{E}}$.   By the Weil pairing, the field $ K_n$ is forced to contain a primitive $p^n$th root of unity $\zeta_{p^n} $ (see for example \cite[Chapter III, Corollary 8.1.1]{Sil}).
Let $G_n = \Gal ( K_n / k )$. As usual, we shall view $\E[p^n]$ as $ \Z/p^n\Z \times \Z/p^n\Z $ and consequently we shall represent $G_n$ as a subgroup of $\textrm{GL}_2 ( \Z / p^n \Z )$, denoted by the same symbol.

As mentioned, the answer to the \emph{Local-Global Divisibility Problem} for $p^n$ is strictly connected to the vanishing condition of the  cohomological  group $H^1(G_n, {\mathcal{E}}[p^n])$ and of the local cohomological group $H^1_{\rm loc}(G_n,{\mathcal{E}}[p^n])$.  Let us recall definitions and results for $\E$.
\par\bigskip

\begin{Def}[Dvornicich, Zannier \cite{DZ}] \label{defloc}Let $\Sigma$ be a group and let $M$ be a $\Sigma$-module.
We say that a cocycle $[c]=[\{Z_{\sigma}\}]\in H^1(\Sigma,M)$ satisfies the \emph{local conditions} if there exists $W_{\sigma}\in M$ such that $Z_{\sigma}=(\sigma - 1)W_{\sigma}$, for all $\sigma\in \Sigma$.
We denote by $H^1_{{\rm loc}}(\Sigma,M)$ the subgroup of $H^1(\Sigma,M)$ formed by such cocycles.
Equivalently, $H^1_{{\rm loc}} ( \Sigma, M )$ is the intersection of the kernels of the restriction maps $H^1 ( \Sigma, M ) \rightarrow H^1 ( C, M )$ as $C$ varies over all cyclic subgroups of $\Sigma$.\end{Def}
\begin{thm}[Dvornicich, Zannier \cite{DZ}] \label{si1}\par \noindent Assume that
$H^1_{\rm loc}(G_n,{\mathcal{E}}[p^n])=0.$
 Let $P\in {\mathcal{E}}(k)$ be a point locally divisible by $p^n$ almost everywhere in the completions $k_v$ of $k$. Then there exists a point $D\in {\mathcal{E}}(k)$, such that $P=p^nD$.
\end{thm}

In \cite{DZ3} they prove that this theorem is not invertible.
Moreover, the remark just after the main theorem and the first few lines of its proof give an intrinsic version of their main theorem  \cite{DZ3}.
\begin{thm}\label{teo:teo11}  Suppose that $\E$ does not admit any $k$-rational isogeny of degree $p$.
Then $H^1 ( G_n, \E[p^n] ) = 0$, for every $n \in \N$.
\end{thm}
Clearly, if the global cohomology $H^1 ( G_n, \E[p^n] ) $ is trivial then also the local cohomology $H^1_{\rm loc}(G_n,{\mathcal{E}}[p^n]).$ So by Theorem \ref{si1}, if $\E$ does not admit any $k$-rational isogeny of degree $p$, then the local-global principle holds for $p^n$.

The following lemma is essentially  proved  in the proof of the Theorem \ref{teo:teo11}, in \cite{DZ3} beginning of page 29.

\begin{lem}\label{lem:lem12}
Suppose that there exists  a nontrivial multiple of the identity $\tau \in G_1$.
Then $H^1 ( G_n, \E[p^n] ) = 0$, for every $n \in \N$.
\end{lem}

\noindent  Another remark along their proof concerns the group $G_1 \cap {\mathpzc{D}}$, where $\mathpzc{D}$  is the subgroup  of  diagonal matrices of $\textrm{GL}_2 ( \F_p )$.

\begin{cor}\label{cor:cor13}
Suppose that $G_1 \cap {\mathpzc{D}}$ is not cyclic. Then $H^1 ( G_n, \E[p^n] ) = 0$, for every $n\in \N$.
\end{cor}
\begin{proof} Since $G_1 \cap {\mathpzc{D}}$ is not cyclic, it contains at least a nontrivial multiple of the identity. Apply Lemma~\ref{lem:lem12}.\end{proof}

\section{Structure of the proof of the Main Theorem} \label{sec1}

 If $H^1 ( G_n, \E[p^n] )=0 $, then also the local cohomology is trivial and, by Theorem \ref{si1}, no counterexample can occur. Therefore we can assume with no restriction that  $H^1 ( G_n, \E[p^n] )\not=0 $. We first describe
the  structure of $G_1$.

\begin{lem}[ \cite{P-R-V} Lemma 7]\label{lem:lem14}
Suppose  that $H^1 ( G_n, \E[p^n] ) \neq 0$. Then either
\begin{equation*}G_1=\langle \rho \rangle  \,\,\,\,\quad {\rm or} \,\,\,\,\quad
   G_1 =\langle \rho,  \sigma \rangle,
\end{equation*}
where $\rho = {\lambda_1\,\,\, 0 \choose 0 \,\,\, \lambda_2}$ is either the identity or a diagonal matrix with $\lambda_1 \neq \lambda_2 \mod (p)$ and $\sigma={1\,\,\, 1 \choose 0\,\,\, 1}$, in a suitable basis of $\E[p]$.
\end{lem}

\begin{proof}  For $n=2$, we have proved the statement in \cite[Lemma 7]{P-R-V}. The proof extends straightforward to a general positive integer $n$. \end{proof}

Note that the order of $\rho$ divides $p-1$ and the order of $\sigma$ is $p$. We also sum up some  immediate, but useful remarks.  From the above description of $G_1$ we directly see that if $\lambda_1=1$, then there exists a torsion point of exact order $p$ defined over $k$.
Indeed the first element of the chosen basis is fixed by both $\rho$ and $\sigma$. In addition, if  $G_1=\langle \rho \rangle$ and  $\lambda_2=1$, then the corresponding eigenvector is a torsion point of exact order $p$ defined over $k$.
In the following, we can exclude these trivial cases and  we denote
\[
\rho =
\left(
\begin{array}{cc}
\lambda_1 & 0 \\
0 & \lambda_2 \\
\end{array}
\right)\,\,\,{\rm with}\,\,\,\lambda_1\not=\lambda_2 \mod (p) \,\,\,{\rm and}\,\,\,\lambda_1\not=1.
\]
Furthermore,  if $G_1$ is cyclic then we assume that  $\lambda_2\neq 1$. \par
\par\bigskip The proof of Theorem \ref{main}
 relies on the following:
\begin{pro}\label{pro:pro63}
Suppose that $H^1( G_n, \E[p^n] ) \neq 0$ and $\rho={\lambda_1 \,\, \,\,\,0 \choose  0 \,\, \,\,\, \lambda_2}$ has order at least $3$. Then we have:
\begin{itemize}
\item[1.] If $\lambda_1\not=1$ and $\lambda_2\not=1$, then $H^1_{{\rm loc}} ( G_n, \E[p^n] ) = 0$;
\item[2.]  If $G_1$ is not cyclic and $\lambda_2= 1$, then $H^1_{{\rm loc}} ( G_n, \E[p^n] ) = 0$.
\end{itemize}
\end{pro}

\bigskip \noindent
The following sections are dedicated to the proof of this proposition. We conclude its proof in section \ref{sec3}.
We now clarify how to deduce Theorem \ref{main}
 from this proposition. In view of Theorem \ref{si1}, our main Theorem is implied by:
\begin{T}
Suppose $ k $ does not contain  $\Q ( \zeta_p + \overline{\zeta_p} )$.
Suppose that $\E$ does not admit any $k$-rational torsion point of exact order $ p $. Then $$H^1_{\rm loc}(G_n,{\mathcal{E}}[p^n])=0.$$
\end{T}

\begin{proof}
If $H^1 ( G_n,\E[p^n] )= 0$, then clearly also the local cohomology is trivial and nothing has to be proven. We may assume $H^1 ( G_n,\E[p^n] )\neq 0$ and we show $H^1_{{\rm loc}} ( G_n, \E[p^n] ) = 0$ using Proposition~\ref{pro:pro63}.
Let $P_1,P_2$ be a basis of $\E[p]$ such that $G_1$ is like in Lemma \ref{lem:lem14}.  First of all we remark that if $k$ does not contain the field $\QQ(\zeta_p+\bar{\zeta_p})$, then the order of $\rho$  is $\geq 3$. In fact,  in this case,
 $[k ( \zeta_p ) : k] \geq 3$.  Recall  that, by Lemma \ref{lem:lem14}, the order of $\rho$ is the largest integer relatively prime to $p$ that divides $\vert G_1 \vert$. In addition $[k ( \zeta_p ) : k] \mid |G_1|$ and
$[k ( \zeta_p ) : k] \mid p-1$. Thus $\rho$ has order $\geq 3$.

 Observe that if $\lambda_1=1$, then $P_1$ is fixed by $G_1$ and therefore $P_1$ is a torsion point of exact order $p$ defined over $ k$. Moreover, if  $G_1$ is cyclic and $\lambda_2=1$, then $P_2$ is  a torsion point of exact order $p$ defined over $ k$. Thus, we can assume $\lambda_1\not=1$ and, furthermore, we can assume that if $\lambda_2=1$, then $G_1$ is not cyclic.
 By Proposition~\ref{pro:pro63}, we get $H^1_{{\rm loc}} ( G_n, \E[p^n] ) = 0$.
\end{proof}

\section{Description of the groups $G_n$}
We are going to choose a suitable basis of $\E[p^n]$. In such a basis, we decompose  $G_n$ by its subgroups of diagonal, strictly
upper triangular and strictly lower triangular matrices. The decomposition in such subgroups and eventually their commutators,
will simplify the study of the cohomology.

We first define some subgroups of $G_n$. Set
\[
L =
\begin{cases}
K_1,&\;\hbox{if} \hspace{5 pt} G_1 = \langle \rho \rangle;\\
K_1^{\langle \sigma \rangle}&\;\hbox{if} \hspace{5 pt} G_1 = \langle \rho, \sigma \rangle.
\end{cases}
\]
Since $\langle \sigma \rangle$ is normal in $G_1$, then $L/ k$ is a cyclic Galois extension.
Its Galois group is generated by a restriction of $\rho$ to $L$.
For every integer  $n$, let $$H_n = \Gal ( K_n/ L ).$$
 Since $L / k$ is Galois, $H_n$ is a normal subgroup of $G_n$.
Moreover it is a $p$-group and $[G_n: H_n]$ is relatively prime to $p$. Thus it is the unique $p$-Sylow subgroup of $G_n$.
We also observe that the exponent of $H_n$ divides $p^n$.
In fact it is isomorphic to a subgroup of $\GL_2 ( \Z / p^n \Z )$ and it is well known that every $p$-Sylow subgroup of $\GL_2 ( \Z / p^n \Z )$ has exponent $p^n$.
Then $$G_n = \langle \rho_n, H_n\rangle,$$ where $\rho_n$ is a lift of $\rho$ to $G_n$.
We first study the lifts $\rho_n$, then we study $H_n$.

\subsection{The lift of $\rho$}
\begin{lem}\label{lem:lem64}
Suppose $ H^1 ( G_n, \E[p^n] )\not=0 $ and  $\rho\not=I$. For any lift $\rho_n$ of $\rho$ to $G_n$, there exists a basis $Q_1,Q_2$ of $ \E[p^n] $, such that $\rho_n$ is diagonal in $G_n$ and the restriction $\rho_j$ of $\rho_n$ to $G_j$ is diagonal with respect to $p^{n-j}Q_1, p^{n-j}Q_2$ .
\end{lem}

\begin{proof}
The proof is by induction. For $n=1$ it is evident. We assume the claim for $n-1$ and we prove it for $n$. Let $\rho_n$ be a lift of $\rho$. Choose a basis ${R_1, R_2}$ of $\E[p^n]$, such that $\{ p R_1, p R_2 \}$ is the basis of $\E[p^{n-1}]$ that diagonalizes the restriction $\rho_{n-1}$ of $\rho_n$ to $G_{n-1}$.
Then $$\rho_n \equiv  \rho_j \mod ( p^j )$$ for $1 \leq j \leq n-1$, with $\rho_j$ as desired. The characteristic polynomial $P(x)$ of $\rho_n$ has integral coefficients. In addition $\lambda_{1,n-1},\lambda_{2,n-1}$ are distinguished roots of $P(x)$ modulo $p$, indeed by inductive hypothesis $\rho_{n-1} \equiv \rho \mod ( p )$. So the first derivate $P^\prime ( \lambda_{i,n-1}  )$ is not congruent to $0$ modulo $( p )$. By Hensel's Lemma, $P(x)$ has roots $\lambda_{i,n}=\lambda_{i,n-1}+t_i p^{n-1}$ with $0\le t_i \le p-1$.
Thus $ \rho_n $ is  diagonalizable in the basis of corresponding eigenvectors and a $p^{n-j}$ multiple gives eigenvectors for a lift of $\rho$ to $G_j$.
\end{proof}
We fix once and for all a basis $\{ Q_1, Q_2 \}$ of $\E[p^n]$ with the properties of the above lemma. Consequently we fix  the basis $\{ p^{n-j} Q_1, p^{n-j} Q_2 \}$ of $\E[p^j]$, for $1 \leq j \leq n-1$. The order of such a lift  of $\rho$ divides $p^{n-1}(p-1)$ and it is divided by the order of $\rho$. Taking an appropriate $p$ power of this lift, we obtain  a diagonal lift $\rho_n$ of $\rho$ such that the order of $\rho_n$ is equal to the order of $\rho$.
\begin{D}
We denote by \[
\rho_n =
\left(
\begin{array}{cc}
\lambda_{1, n} & 0 \\
0 & \lambda_{2, n} \\
\end{array}
\right)
\]   a diagonal lift of $\rho$ to $G_n$ of the same order than $\rho$.
\end{D}

\begin{rem}\label{inversa} Assume  that $H^1(G_n,\E[p^n])\not=0$ and that  $\rho$ has order $\geq 3$. Then $\lambda_{2, n} \lambda_{1, n}^{-1} - \lambda_{1, n} \lambda_{2, n}^{-1}$ is invertible  (where $\lambda_{i, n}^{-1}$ is the inverse of $\lambda_{i, n}$ in $( \Z / p^n \Z )^\ast$).
Indeed, if  $\lambda_{2, n} \lambda_{1, n}^{-1} - \lambda_{1, n} \lambda_{2, n}^{-1} \equiv 0 \mod ( p )$, then $\lambda_{1, n}^2 \equiv \lambda_{2, n}^2 \mod ( p )$.
Thus $\lambda_1^2 \equiv \lambda_2^2  \mod ( p )$ and $\rho^2$ is a scalar multiple of the identity. But in view of Corollary \ref{cor:cor13}, only the identity is such a multiple in $G_1$. Then $\rho^2=I$, which is a contradiction.
\end{rem}

\subsection{The decomposition of $G_n$}

We consider the following subgroups of $\textrm{GL}_2(\Z/p^n\Z)$:\\
    the  subgroup ${\mathpzc{sU}}$ of strictly upper triangular matrices;\\
     the subgroup  ${\mathpzc{sL}}$ of strictly lower triangular matrices;\\
      the      subgroup  ${\mathpzc{D}}$  of   diagonal   matrices.\\

\noindent
We decompose $H_n=\textrm{Gal} ( K_n/ L )$ in products of diagonal, strictly upper triangular and strictly lower triangular matrices. Then the group $G_n$ has a similar decomposition, as it is generated by $\rho_n$ and $H_n$.

\begin{pro}\label{nuovo1}
Assume that $ H^1 ( G_n, \E[p^n] )\not=0 $ and that the order of $\rho$ is at least $3$. Then, the group  $H_n$ is generated by matrices of ${\mathpzc{D}}_n=H_n \cap {\mathpzc{D}}$, ${\mathpzc{sU}}_n= H_n\cap {\mathpzc{sU}}$ and ${\mathpzc{sL}}_n=H_n \cap {\mathpzc{sL}}$.
\end{pro}

The proof of  Proposition \ref{nuovo1} is done by induction. The structure is technical and we do it along several steps.
Recall  that  $H_n$ restricts to either the identity, or $\langle {1\,\,\, 1 \choose 0\,\,\, 1}\rangle$ modulo $p$. Therefore every matrix of $H_n$  has the form
\[
\left(
\begin{array}{cc}
1 +p a & e \\
p c & 1 + p d \\
\end{array}
\right),
\]
with $e\in  \Z / p^n \Z$ and $a, c, d \in \Z / p^{n-1} \Z$. Of course every matrix can be decomposed as product of diagonal, strictly upper triangular and strictly lower triangular matrices. Here we shall prove that for $\tau\in H_n$ such factors are in $H_n$ as well.
So, we shall prove that certain matrices are in $H_n$. Since $H_n$  is normal in $G_n$, then for every $\tau\in H_n$, also $\rho_n^i \tau^m\rho_n^{-i} \in H_n$,  for every integer $i,m$. Besides, we recall that $H_n$ has  exponent dividing $p^n$ and therefore powers of $\tau$ are well defined for classes $m\in \Z / p^n \Z$.
Other useful matrices in $H_n$ are constructed  in  the following:
\begin{proprieta}\label{Propr}   Assume that  $ H^1 ( G_n, \E[p^n] )\not=0 $ and $\rho = {\lambda_1 \,\, \,\,\,0 \choose  0 \,\, \,\,\, \lambda_2}$ has order at least $3$.
 Let $H^*_{n} = \Gal ( K_n / K_{n-1} )\subset H_n$. Suppose that
\[
\tau = \left(
\begin{array}{cc}
1 + p^{n-1} a & p^{n-1} b \\
p^{n-1} c & 1 + p^{n-1} d \\
\end{array}
\right) \in H^*_{n} ,
\]
for certain $a, b, c, d \in \Z / p \Z$.
Then
\[
\left(
\begin{array}{cc}
1 + p^{n-1} a & 0\\
0 & 1 + p^{n-1} d \\
\end{array}
\right), \
\left(
\begin{array}{cc}
1 & p^{n-1} b\\
0 & 1 \\
\end{array}
\right), \
\left(
\begin{array}{cc}
1 & 0\\
p^{n-1} c & 1 \\
\end{array}
\right) \in H^*_{n} .
\]
\end{proprieta}
\begin{proof}
We recall the following property of basic linear algebra.
Let $ V $ be a vector space  and $W$ be a subspace of $V$.
Let $ \phi $ be an automorphism of $ V $ such that $\phi ( W ) = W$.
Let $v_1, \ldots , v_m \in V$ be eigenvectors of $\phi$ for distinct eigenvalues. If
$
v_1 + \ldots + v_m \in W,
$
then $v_i \in W$, for all $i$.

Let $V_{n}$ be the multiplicative subgroup of $\GL_2 ( \Z /p^n \Z )$ of the matrices congruent to $I$ modulo $p^{n-1}$. With the scalar multiplication given by taking  powers, the group $V_n$ is a $\Z / p \Z$-vector space of dimension $4$.

Observe that $H^*_{n} = \Gal ( K_n / K_{n-1} )$ is a vector subspace of $ V_{n} $.
The map $\phi_{n} \colon V_{n}  \rightarrow V_{n} \,;\,\tau \to \rho_n \tau \rho_n^{-1}$ is an automorphism. Since $H^*_{n} $ is normal in $G_n$, then $\phi_{n} ( H^*_{n}  ) = H^*_{n} $.
By a simple verification we can determine a basis of eigenvectors for $\phi_n$. To  the eigenvalue $1$ correspond the two eigenvectors
{\scriptsize $\left(\begin{array}{cc}
 1+ p^{n-1} & 0 \\
 0 &  1\\
\end{array}
\right)$,
 $\left(\begin{array}{cc}
 1 & 0 \\
 0 &  1+p^{n-1}\\
\end{array}
\right)$}; to    the eigenvalue  $\lambda_{1}  \lambda_{2}^{-1}$  corresponds the eigenvector {\scriptsize $\left(\begin{array}{cc}
 1 & p^{n-1} \\
 0 &  1\\
\end{array}
\right)$}; and to the   eigenvalue $\lambda_{2}  \lambda_{1}^{-1}$ corresponds the eigenvector {\scriptsize $\left(\begin{array}{cc}
 1 & 0 \\
 p^{n-1} &  1\\
\end{array}
\right)$}. Note that, by Remark \ref{inversa}, the last two eigenvalues are distinct. Applying the above  result from linear algebra to $ V_{n}$,  $ H^*_{n} $ and $ \phi_{n}$, we obtain the desired result.
\end{proof}

We are now ready to prove a property that represents the inductive step for the wished decomposition of $H_n$.
\begin{proprieta}\label{pro:pro66}
Assume that $ H^1 ( G_n, \E[p^n] )\not=0 $ and that $\rho = {\lambda_1 \,\, \,\,\,0 \choose  0 \,\, \,\,\, \lambda_2}$ has order at least $3$.
\begin{itemize}
\item[i.] Suppose
\[
\tau= \left(
\begin{array}{cc}
1 + p a & p^{n-1} b \\
p^{n-1}c & 1 + p d \\
\end{array}
\right)\in H_n,
\] with $a, d \in  \Z / p^{n-1} \Z$ and $b,c \in \Z/ p \Z$. Then
\[
\left(
\begin{array}{cc}
1  & p^{n-1} b \\
0 & 1  \\
\end{array}
\right), \
\left(
\begin{array}{cc}
1  & 0 \\
p^{n-1}c & 1  \\
\end{array}
\right)\in H_n.
\]
Consequently, $\tau$ decomposes in $H_n$ as
\[
\tau=
\left(
\begin{array}{cc}
1 + p a & 0 \\
0 & 1 + p d \\
\end{array}
\right)\left(
\begin{array}{cc}
1 & p^{n-1} b \\
0 & 1 \\
\end{array}
\right)
\left(
\begin{array}{cc}
1 & 0 \\
p^{n-1} c & 1 \\
\end{array}
\right).
\]
\item[ii.]
Suppose
\[\tau=
\left(
\begin{array}{cc}
1+p^{n-1} a & p^{n-1} b \\
pc & 1 + p^{n-1} d \\
\end{array}
\right)\in H_n
\]
with $a, b ,d \in  \Z/ p \Z$ and $c \in \Z/ p^{n-1} \Z$. Then
\[
\left(
\begin{array}{cc}
1 + p^{n-1} a & 0 \\
0 & 1 + p^{n-1} d  \\
\end{array}
\right), \
\left(
\begin{array}{cc}
1 & p^{n-1} b \\
0 & 1  \\
\end{array}
\right) \in H_n.
\]
Consequently, $\tau$ decomposes in $H_n$ as
\[\tau=
\left(
\begin{array}{cc}
1 &0 \\
pc & 1 \\
\end{array}
\right)
\left(
\begin{array}{cc}
1+ p^{n-1} a & 0 \\
0 & 1+p^{n-1} d \\
\end{array}
\right)
\left(
\begin{array}{cc}
1 & p^{n-1} b \\
0 & 1 \\
\end{array}
\right).
\]
\item[iii.] Suppose
\[
\tau= \left(
\begin{array}{cc}
1 + p^{n-1} a & e \\
p^{n-1} c & 1 + p^{n-1} d \\
\end{array}
\right)\in H_n,
\] with $e \in  \Z / p^n \Z$ and $a, c, d \in \Z/ p \Z$. Then
\[
\left(
\begin{array}{cc}
1 + p^{n-1} ( a - e c )  & 0 \\
0 & 1 + p^{n-1} d  \\
\end{array}
\right), \
\left(
\begin{array}{cc}
1  & - p^{n-1} e d  \\
0 & 1  \\
\end{array}
\right), \
\left(
\begin{array}{cc}
1 & 0 \\
p^{n-1} c & 1  \\
\end{array}
\right)\in H_n.
\]
Consequently, $\tau$ decomposes in $H_n$ as
\[
\tau=
\left(
\begin{array}{cc}
1 & e \\
0 & 1 \\
\end{array}
\right)
\left(
\begin{array}{cc}
1 + p^{n-1} ( a - e c )  & 0 \\
0 & 1 + p^{n-1} d  \\
\end{array}
\right)
\left(
\begin{array}{cc}
1  & - p^{n-1} e d   \\
0 & 1  \\
\end{array}
\right)
\left(
\begin{array}{cc}
1 & 0 \\
p^{n-1} c & 1  \\
\end{array}
\right).
\]
\end{itemize}

\end{proprieta}
\begin{proof} Recall that $\tau^m$ is well defined for classes $m \in \Z/p^n \Z$.\par
Part i.  Observe that
\[
\tau =
\left(
\begin{array}{cc}
1 + p a & 0 \\
0 & 1 + p d \\
\end{array}
\right)\left(
\begin{array}{cc}
1 & p^{n-1} b \\
p^{n-1} c & 1 \\
\end{array}
\right).
\]
We are going to show that the second matrix of the product is in $H_n$ and so must be the other. Since $H_n$ is normal in $G_n$, the matrix $\rho_n \tau \rho_n^{-1} \tau^{-1}\in H_n$.
A tedious but  simple computation gives
\[
\rho_n \tau \rho_n^{-1} \tau^{-1} =
\left(
\begin{array}{cc}
1 & \hspace{0.1cm} p^{n-1} b ( \lambda_{1, n}  \lambda_{2, n}^{-1} - 1 ) \\
p^{n-1} c ( \lambda_{2, n}  \lambda_{1, n}^{-1} - 1 ) & 1 \\
\end{array}
\right) \in H_n.
\]
This matrix is in $H^*_n$, indeed it reduces to the identity modulo $p^{n-1}$.
Applying  Property~\ref{Propr} we get
\[
\left(
\begin{array}{cc}
1 & p^{n-1} b ( \lambda_{1, n}  \lambda_{2, n}^{-1} - 1 ) \\
0 & 1 \\
\end{array}
\right), \
\left(
\begin{array}{cc}
1 & 0 \\
p^{n-1} c ( \lambda_{2, n}  \lambda_{1, n}^{-1} - 1 ) & 1 \\
\end{array}
\right) \in H^*_n \subseteq H_n.
\]
By remark \ref{inversa}, we know $\lambda_{1, n} \not \equiv \lambda_{2, n} \mod ( p )$.  So $( \lambda_{1, n}  \lambda_{2, n}^{-1} - 1 )$ and  $( \lambda_{2, n}  \lambda_{1, n}^{-1} - 1 )$ are invertible in $\Z/p^n\Z$. Taking the associated inverse power, we obtain
\[
{1\,\,\,p^{n-1} b  \choose 0\,\,\, 1 } \quad {\rm and } \quad {1\,\,\,0 \choose p^{n-1} c\,\,\, 1 } \in H^*_{n}  \subset H_n.\]
Thus
$\tau$ decomposes in $H_n$ as
\[
\tau=
\left(
\begin{array}{cc}
1 + p a & 0 \\
0 & 1 + p d \\
\end{array}
\right)\left(
\begin{array}{cc}
1 & p^{n-1} b \\
0 & 1 \\
\end{array}
\right)
\left(
\begin{array}{cc}
1 & 0 \\
p^{n-1} c & 1 \\
\end{array}
\right).
\]

Part ii.  This proof is similar to the previous one.
Observe
\[
\tau =
\left(
\begin{array}{cc}
1 & 0 \\
p c & 1 \\
\end{array}
\right)\left(
\begin{array}{cc}
1 + p^{n-1} a & p^{n-1} b \\
0 & 1 + p^{n-1} d \\
\end{array}
\right).
\]

\noindent  By induction, we can prove
$$\tau^{\lambda_{2,n}\lambda_{1,n}^{-1}}=\left(
\begin{array}{cc}
1 + p^{n-1} a \lambda_{2, n} \lambda_{1, n}^{-1} & p^{n-1} b  \lambda_{2,n}\lambda_{1,n}^{-1} \\
p c \lambda_{2,n}\lambda_{1,n}^{-1} & 1+  p^{n-1} d \lambda_{2, n} \lambda_{1, n}^{-1}  \\
\end{array}
\right).$$
In addition
\[
\rho_n \tau \rho_n^{-1} \tau^{- {\lambda_{2,n}\lambda_{1,n}^{-1}}} =
\left(
\begin{array}{cc}
1 + p^{n-1} a ( 1 - \lambda_{2, n} \lambda_{1, n}^{-1} ) & p^{n-1} b ( \lambda_{1, n}  \lambda_{2, n}^{-1} - \lambda_{2, n} \lambda_{1, n}^{-1} ) \\
0 & 1 + p^{n-1} d ( 1 - \lambda_{2, n}  \lambda_{1, n}^{-1} ) \\
\end{array}
\right) \in H_n.
\]  As this matrix is in $H_n$ and reduces to the identity mod $p^{n-1}$, it is also in $H^*_n$.
Recall that $\lambda_{1, n}^2 \not \equiv \lambda_{2, n}^2 \mod ( p )$. Applying Property~\ref{Propr}  and taking appropriated powers, we get \[\left(
\begin{array}{cc}
1 + p^{n-1} a & 0 \\
0 & 1+p^{n-1}d \\
\end{array}
\right) \quad  { \rm and } \quad \left(
\begin{array}{cc}
1  & p^{n-1}b \\
0 & 1  \\
\end{array}
\right)\in H^*_{n}  \subset H_n.\]

\noindent Thus $\tau$ decomposes in $H_n$ as
\[
\tau=
\left(
\begin{array}{cc}
1 & 0 \\
p c & 1 \\
\end{array}
\right)\left(
\begin{array}{cc}
1 + p^{n-1} a & 0 \\
0 & 1 + p^{n-1} d \\
\end{array}
\right)
\left(
\begin{array}{cc}
1 & p^{n-1} b \\
0 & 1 \\
\end{array}
\right).
\]

Part iii. If $p^{n-1} \mid e$, then $\tau \in H^*_{n}$ and the assertion follows from Property~\ref{Propr}. Therefore we can assume that $p^{n-1}$ does not divide $e$.
We compute $\tau^{\lambda_{2, n} \lambda_{1, n}^{-1} }$ which is an element of $H_n$.
 By induction, we get \[
\tau^{\lambda_{2, n} \lambda_{1, n}^{-1} }=
\left(
\begin{array}{cc}
1 + p^{n-1} a'' & e \lambda_{2,n}\lambda_{1,n}^{-1} + p^{n-1} b''\\
p^{n-1} \lambda_{2,n}\lambda_{1,n}^{-1} c  & 1 + p^{n-1} d'' \\
\end{array}
\right),
\]
with $a'',b'',d'' \in \Z / p \Z$.
As  $H_n$ is normal and $\lambda_{2, n} \lambda_{1, n}^{-1}$ is invertible, the following matrix is in $H_n$. We have
$$\gamma_1 = \rho_n \tau \rho_n^{-1} \tau^{-\lambda_{2, n} \lambda_{1, n}^{-1}}
 =
\left(
\begin{array}{cc}
1 + p^{n-1} a^\prime & e ( \lambda_{1, n}  \lambda_{2, n}^{-1} - \lambda_{2, n}  \lambda_{1, n}^{-1} ) + p^{n-1} b^\prime  \\
0  & 1 + p^{n-1} d^\prime
\end{array}
\right) \in H_n,$$
where $a^\prime, b^\prime, d^\prime \in \Z / p \Z$.
As $H_n$ is normal, also the following matrix is an element of $H_n$:
\[
\gamma_2 = \rho_n \gamma_1 \rho_n^{-1} \gamma_1^{-1} =
\left(
\begin{array}{cc}
1 & \hspace{0.1cm} e ( \lambda_{1, n}  \lambda_{2, n}^{-1} - \lambda_{2, n}  \lambda_{1, n}^{-1} ) ( \lambda_{1, n}  \lambda_{2, n}^{-1} - 1 )  + p^{n-1} e^\prime\\
0 & 1 \\
\end{array}
\right),
\]
where $e^\prime \in \Z / p \Z$.
Recall that  $l= ( \lambda_{1, n}  \lambda_{2, n}^{-1} - \lambda_{2, n}  \lambda_{1, n}^{-1} ) ( \lambda_{1, n}  \lambda_{2, n}^{-1} - 1 ) $ is invertible and that $p^{n-1}$ does not divide $e$.  So $e=p^r f$ with $f$ coprime to $p$ and $r<n-1$. Then $\lambda=l+p^{n-r-1}e'f^{-1}$ is invertible in $\Z/p^n\Z$ and $\gamma_2^{\lambda^{-1}}={1\,\,\, e \choose 0\,\,\, 1} $. Thus ${1\,\,\, e \choose 0\,\,\, 1} $ is  in  $ H_n$.
A simple computation gives \[
\tau=
\left(
\begin{array}{cc}
1 & e \\
0 & 1 \\
\end{array}
\right)
\left(
\begin{array}{cc}
1 + p^{n-1} ( a - e c )  & - p^{n-1} e d \\
p^{n-1} c & 1 + p^{n-1} d  \\
\end{array}
\right).
\] The second matrix is in $H^*_n$.
By Property~\ref{Propr}, the matrices  ${1 + p^{n-1} ( a - e c )\hspace{0.2cm} 0  \choose 0\hspace{0.2cm} 1 + p^{n-1} d }$, ${1\hspace{0.2cm} - p^{n-1} e d \choose 0 \hspace{0.2cm} 1 }$, ${1 \hspace{0.2cm} 0  \choose p^{n-1} c  \hspace{0.2cm} 1 } $ are in $  H^*_{n} $ and consequently in $ H_n$.
Thus $\tau$ decomposes in $H_n$ as
\[
\tau=
\left(
\begin{array}{cc}
1 & e \\
0 & 1 \\
\end{array}
\right)\left(
\begin{array}{cc}
1 + p^{n-1} ( a - e c )  & 0 \\
0 & 1 + p^{n-1} d \\
\end{array}
\right)
\left(
\begin{array}{cc}
1 & - p^{n-1} e d \\
0 & 1 \\
\end{array}
\right)
\left(
\begin{array}{cc}
1 & 0 \\
p^{n-1} c & 1 \\
\end{array}
\right).
\]
\end{proof}
The above property is exactly the inductive step to decompose $H_n$ as product of diagonal, strictly upper triangular  and strictly lower triangular  matrices.

\begin{proof}[Proof of Proposition \ref{nuovo1}]
Proceed by induction. For $H_1$, which is either the identity or  $\langle {1\,\,\, 1 \choose 0\,\,\, 1}\rangle$, the claim  is clear. Suppose that $H_r$ is decomposable as product of such matrices for $r< n$. We show it for $n$.
Let $\tau$ be a matrix in $H_n$. Then the reduction $\tau_{n-1}$ of $\tau$ mod $p^{n-1}$ is a product $\prod \delta_i$ of diagonal, strictly upper triangular and strictly lower triangular matrices. Consider lifts $\tilde{\delta}_i$ of $\delta_i$ to $H_n$. Then $\tau= \tilde{\delta}  \prod \tilde{\delta}_i$, where $\tilde{\delta}$ reduced to the identity mod  $p^{n-1}$. Therefore it is sufficient to prove the assertion for a matrix reducing to a diagonal, to a strictly upper triangular or to a strictly lower triangular matrix mod $p^{n-1}$.
By  Property  \ref{pro:pro66}, the matrix $\tau $ can be decomposed in the desired product.
\end{proof}

\subsection{Some commutators}

To study the cohomology we still need to describe the commutators of some of the subgroups of $G_n$.
We denote by ${\mathpzc{U}}$  the subgroup of  upper triangular matrices  of $\textrm{GL}_2(\Z/p^n\Z)$ and by
 ${\mathpzc{L}}$ the  subgroup of   lower   triangular   matrices of $\textrm{GL}_2(\Z/p^n\Z)$.
\begin{lem}
\label{nuovo2}
The commutators  ${\mathpzc{U}}_n^\prime$ and ${\mathpzc{L}}_n^\prime$ of the groups ${\mathpzc{U}}_n= G_n \cap {\mathpzc{U}}$ and ${\mathpzc{L}}_n=G_n \cap {\mathpzc{L}}$ in $G_n$ are cyclic.
If in addition $ H^1 ( G_n, \E[p^n] )\not=0 $, the order of $\rho$ is at least $3$ and $G_1$ is not cyclic, then \[{\mathpzc{U}}_n^\prime=
\bigg \langle\left(
\begin{array}{cc}
1 & 1 \\
0 & 1 \\
\end{array}
\right)\bigg \rangle.
\]
\end{lem}
\begin{proof}
The commutator subgroup ${\mathpzc{U}}_n^\prime$ is generated by the elements $\delta \gamma \delta^{-1} \gamma^{-1}$,
with \[
\delta =
\left(
\begin{array}{cc}
a & b \\
0 & d \\
\end{array}
\right),
\gamma =
\left(
\begin{array}{cc}
a^\prime & b^\prime \\
0 & d^\prime \\
\end{array}
\right)
\in {\mathpzc{U}}_n,
\]
where the entries are  in $ \Z/ p^n \Z$ and  the elements on the diagonals are invertible modulo $ p^n $.
A short computation shows that
\begin{equation}\label{eqn:rel61}
\delta \gamma \delta^{-1} \gamma^{-1} =
\left(
\begin{array}{cc}
1 & ( a b^\prime - a^\prime b + b d^\prime - b^\prime d ) d^{-1} d^{{\prime}^{-1}}  \\
0 & 1 \\
\end{array}
\right).
\end{equation}
Then
\[{\mathpzc{U}}_n^\prime=
\bigg \langle\left(
\begin{array}{cc}
1 & p^j \\
0 & 1 \\
\end{array}
\right)\bigg \rangle,
\]
 for an integer $j \in \N$.
The proof for ${\mathpzc{L}}_n^\prime$ is analogous.

Suppose that in addition $H^1(G_n,\E[p^n])\neq 0$, the order of $\rho$ is at least $3$ and $G_1$ is not cyclic. Then $\sigma= { 1\,\,\, 1 \choose 0\,\,\, 1}\in G_1$. Let $\sigma_n$ be a lift of $\sigma$ to $G_n$. By Proposition~\ref{nuovo1}, $\sigma_n$ decomposes as a product of diagonal, strictly upper triangular and strictly lower triangular matrices. Since $\sigma_n$ does not restrict to a diagonal matrix,  at least one of its factors is of the type \[
\delta =
\left(
\begin{array}{cc}
1 & b \\
0 & 1 \\
\end{array}
\right) \,\,\,\,{\rm with}\,\,\,\,  b \not\equiv 0 \mod ( p ).
\]

\noindent A simple computation gives
\[
\rho_n \delta \rho_n^{-1} \delta^{-1} =
\left(
\begin{array}{cc}
1 & ( \lambda_{1, n} \lambda_{2, n}^{-1} - 1 ) b \\
0 & 1 \\
\end{array}
\right)\in {\mathpzc{U}}_n^\prime.
\]
Recall that $\lambda_{1, n} \not \equiv \lambda_{2, n} \mod ( p )$ and $b \not \equiv 0 \mod ( p )$. Then a power of this matrix is $ {1\,\,\, 1 \choose 0\,\,\, 1} \in {\mathpzc{U}}_n^\prime$.

\end{proof}

\section{Proof of Proposition 8} \label{sec3}
In this section we study the cohomology of $G_n$. We recall a useful classical lemma.

\begin{lem}[Sah Theorem,~\cite{Lan} Theorem 5.1]\label{lem:lem61}
Let $\Sigma$ be a group and let $M$ be a $\Sigma$-module.
Let $\alpha$ be in the center of $\Sigma$.
Then $H^1 (\Sigma, M )$ is annihilated by the map $x \rightarrow \alpha x - x$ on $M$.
In particular, if this map is an automorphism of $M$, then $H^1 ( \Sigma, M ) = 0$.
\end{lem}

In the following proposition, we study the relation between the  eigenvalues of $\rho = {\lambda_1\,\,\, 0 \choose 0\,\,\,\,\,\,\lambda_2}$ and the triviality of  the local cohomology  of certain subgroups of $G_n$.  Such subgroups have intersections that allows us to glue those cohomologies together and to deduce the triviality of the local cohomology of $G_n$.

\begin{pro}\label{teo:teo62}
Assume  that $H^1( G_n, \E[p^n] ) \neq 0$ and that $\rho = {\lambda_1 \,\, \,\,\,0 \choose  0 \,\, \,\,\, \lambda_2}$ has order at least $3$.  Then,  we have:
\begin{enumerate}
 \item The groups $H^1_{{\rm loc}} ( \langle \rho_n, {\mathpzc{sU}}_n \rangle, \E[p^n] )$ and $H^1_{{\rm loc}} ( \langle \rho_n, {\mathpzc{sL}}_n \rangle, \E[p^n] )$ are trivial;
\item If $ \lambda_1 \neq 1 $ and $ \lambda_2 \neq 1 $, then $H^1_{{\rm loc}} ( {\mathpzc{U}}_n, \E[p^n] )$  and $H^1_{{\rm loc}} ( {\mathpzc{L}}_n, \E[p^n] )$   are trivial;
\item If $G_1$ is not cyclic and $ \lambda_2 = 1 $, then $H^1_{{\rm loc}} ( {\mathpzc{U}}_n, \E[p^n] ) = 0$.
\end{enumerate}
\end{pro}

\begin{proof}

Part 1. We prove the triviality of $H^1_{{\rm loc}} ( \langle \rho_n, {\mathpzc{sL}}_n \rangle, \E[p^n] )$. The triviality of $H^1_{{\rm loc}} ( \langle \rho_n, {\mathpzc{sU}}_n \rangle, \E[p^n] )$  is similar and it is left to the reader.

Remark that ${\mathpzc{sL}}_n$ is cyclic generated  by ${1\,\,\, 0 \choose p^j\,\,\, 1}$ where  $p^j$  is the minimal power of $p$ dividing all the entries $c$ of any matrix ${1\,\,\, 0 \choose c\,\,\, 1} \in G_n$. Then, we immediately get $H^1_{{\rm loc}} ( {\mathpzc{sL}}_n, \E[p^n] ) = 0$.
Moreover ${\mathpzc{sL}}_n$ is a normal subgroup of $\langle \rho_n, {\mathpzc{sL}}_n\rangle$.
The order of $\rho_n$ is equal to the order of $\rho$, which is relatively prime to $p$. Thus  ${\mathpzc{sL}}_n$ is the $p$-Sylow subgroup of $\langle \rho_n, {\mathpzc{sL}}_n\rangle$.
By~\cite[Proposition 2.5]{DZ},  if  $H^1_{{\rm loc}} ( {\mathpzc{sL}}_n, \E[p^n] ) = 0$ then $H^1_{{\rm loc}} (  \langle \rho_n, {\mathpzc{sL}}_n \rangle, \E[p^n] ) = 0$.

\bigskip

Part 2.
We only present the proof for ${\mathpzc{U}}_n$. The proof for ${\mathpzc{L}}_n$ is similar.
Since ${\mathpzc{U}}_n^\prime$ is normal, we have the inflaction-restriction sequence:
\[
0 \rightarrow H^1 ( {\mathpzc{U}}_n / {\mathpzc{U}}_n^\prime, \E[p^n]^{{\mathpzc{U}}_n^\prime} ) \rightarrow H^1 ( {\mathpzc{U}}_n, \E[p^n] ) \rightarrow H^1 ( {\mathpzc{U}}_n^\prime, \E[p^n] ).
\]
The matrix $\rho_n \in {\mathpzc{U}}_n$  is diagonal and  modulo $p$ reduces to $\rho$. By hypothesis, the eigenvalues $\lambda_1, \lambda_2$ of $\rho$ are both different from $1$. Thus  $\rho_n - I$ is an isomorphism of $\E[p^n]$ to itself.
Let $[\rho_n]$ be the class of $\rho_n$ in ${\mathpzc{U}}_n / {\mathpzc{U}}_n^\prime$. Then $[\rho_n] - I$ is an isomorphism of $\E[p^n]^{{\mathpzc{U}}_n^\prime}$ into itself.
Since ${\mathpzc{U}}_n / {\mathpzc{U}}_n^\prime$ is abelian, then
by Lemma~\ref{lem:lem61} \begin{equation}\label{eqn:rel62}
H^1 ( {\mathpzc{U}}_n / {\mathpzc{U}}_n^\prime, \E[p^n]^{{\mathpzc{U}}_n^\prime} ) = 0.\end{equation}

On the other hand, by Lemma~\ref{nuovo2}, ${\mathpzc{U}}_n^\prime$ is cyclic. Moreover $H^1_{{\rm loc}} ( {\mathpzc{U}}_n, \E[p^n] )$ is the intersection of the kernels of the restriction maps $H^1 ( {\mathpzc{U}}_n, \E[p^n] ) \rightarrow H^1 ( C, \E[p^n] )$, as $C$ varies over all cyclic subgroups of ${\mathpzc{U}}_n$ (see Definition \ref{defloc}).
If $H^1_{{\rm loc}} ( {\mathpzc{U}}_n, \E[p^n] ) \neq 0$, then
$
H^1 ( {\mathpzc{U}}_n / {\mathpzc{U}}_n^\prime, \E[p^n]^{{\mathpzc{U}}_n^\prime} ) \neq 0
$,
which contradicts~(\ref{eqn:rel62}).  So  $H^1_{{\rm loc}} ( {\mathpzc{U}}_n, \E[p^n] ) = 0$.

\bigskip

Part 3. As ${\mathpzc{U}}_n^\prime$ is normal in ${\mathpzc{U}}_n$, we consider the inflaction-restriction sequence
\[
0 \rightarrow H^1 ( {\mathpzc{U}}_n / {\mathpzc{U}}_n^\prime, \E[p^n]^{{\mathpzc{U}}_n^\prime} ) \rightarrow H^1 ( {\mathpzc{U}}_n, \E[p^n] ) \rightarrow H^1 ( {\mathpzc{U}}_n^\prime, \E[p^n] ).
\]

 Recall that $Q_1$ and $Q_2$ is the  basis of $\E[p^n]$ such that $\rho_n=${\scriptsize $\left(
                                                            \begin{array}{cc}
                                                              \lambda_{1,n}&0 \\
                                                              0&\lambda_{2,n} \\
                                                            \end{array}
                                                          \right)$}.
Then $\rho_n ( Q_1 ) = \lambda_{1, n} Q_1$ and  $\rho_n ( Q_2 ) = \lambda_{2, n} Q_2$.
We first prove that $\E[p^n]^{{\mathpzc{U}}_n^\prime} \subseteq \langle Q_1 \rangle$.
Let $a, b \in \Z / p^n \Z$ and let $a Q_1 + b Q_2 \in \E[p^n]^{{\mathpzc{U}}_n^\prime}$.
By Lemma~\ref{nuovo2},  $ {1\,\,\, 1 \choose 0\,\,\, 1} \in {\mathpzc{U}}_n^\prime$.
Then
\[
a Q_1 + b Q_2 = \sigma_n ( a Q_1 + b Q_2 ) = ( a + b ) Q_1 + b Q_2.
\]
Thus $b Q_1=0$. Whence $b = 0$, because
 $Q_1$ has exact order $p^n$.

Let $[\rho_n]$ be the class of $\rho_n$ in ${\mathpzc{U}}_n / {\mathpzc{U}}_n^\prime$.
Since $\rho$ has order at least $3$ and $\lambda_2 = 1$, then $\lambda_1\not=1$. Thus $(\rho_n - I )Q_1  = ( \lambda_{1, n} - 1 ) Q_1$  with $\lambda_{1, n} \not \equiv 1 \mod ( p )$. Consequently the restriction of $\rho_n - I$ to $\langle Q_1 \rangle$ is an isomorphism.
As $\E[p^n]^{{\mathpzc{U}}_n^\prime} \subseteq \langle Q_1 \rangle$,
also $[\rho_n] - I$ is an isomorphism of $\E[p^n]^{{\mathpzc{U}}_n^\prime}$ to itself.
Moreover ${\mathpzc{U}}_n/ {\mathpzc{U}}_n^\prime$ is abelian.
By Lemma~\ref{lem:lem61},
\begin{equation}\label{eqn:rel162}
H^1 ( {\mathpzc{U}}_n / {\mathpzc{U}}_n^\prime, \E[p^n]^{{\mathpzc{U}}_n^\prime} ) = 0.
\end{equation}

On the other hand,  ${\mathpzc{U}}_n^\prime$ is cyclic and $H^1_{{\rm loc}} ( {\mathpzc{U}}_n, \E[p^n] )$ is the intersection of the kernels of the restriction maps $H^1 ( {\mathpzc{U}}_n, \E[p^n] ) \rightarrow H^1 ( C, \E[p^n] )$, as $C$ varies over all cyclic subgroups of ${\mathpzc{U}}_n$ (see Definition \ref{defloc}).
If  $H^1_{{\rm loc}} ( {\mathpzc{U}}_n, \E[p^n] ) \neq 0$, then
$H^1 ( {\mathpzc{U}}_n / {\mathpzc{U}}_n^\prime, \E[p^n]^{{\mathpzc{U}}_n^\prime} ) \neq 0.$ This contradicts (\ref{eqn:rel162}). So $H^1_{{\rm loc}} ( {\mathpzc{U}}_n, \E[p^n] ) = 0$.

\end{proof}

We are now ready to conclude the proof of Proposition~\ref{pro:pro63}. The core idea is to glue the cohomology of the subgroups of $G_n$ via some special elements in their  intersections. In the previous part we already proved that the local cohomology of such subgroups is trivial and so also the local cohomology of $G_n$ is trivial.
For the convenience of the reader, we recall the statement:

\begin{proposi}
Suppose that $H^1( G_n, \E[p^n] ) \neq 0$ and that $\rho={\lambda_1 \,\, \,\,\,0 \choose  0 \,\, \,\,\, \lambda_2}$ has order at least $3$.
\begin{itemize}
\item[1.] If $\lambda_1\not=1$ and $\lambda_2\not=1$, then $H^1_{{\rm loc}} ( G_n, \E[p^n] ) = 0$;
\item[2.]  If $G_1$ is not cyclic and $\lambda_2= 1$, then $H^1_{{\rm loc}} ( G_n, \E[p^n] ) = 0$.
\end{itemize}
\end{proposi}

\begin{proof}

Part 1.
 Consider the restrictions
\[
r_L : H^1 ( G_n, \E[p^n] ) {\rightarrow} H^1 ( {\mathpzc{L}}_n, \E[p^n] ), \]
\[ r_U : H^1 ( G_n, \E[p^n] ) {\rightarrow} H^1 ( {\mathpzc{U}}_n, \E[p^n] ).
\]
Let $Z$ be a cocycle from $G_n$ to $\E[p^n]$, such that its class $[Z] \in H^1_{{\rm loc}} ( G_n, \E[p^n] )$. If a cocycle satisfies the local conditions relative to $G_n$ (see Definition \ref{defloc}), then it satisfies them relative to any subgroup of $G_n$. Thus  $r_L([Z]) \in H^1_{{\rm loc}} ({\mathpzc{L}}_n, \E[p^n] ) $ and $r_U([Z]) \in H^1_{{\rm loc}} ( {\mathpzc{U}}_n, \E[p^n] ).$  By Proposition~\ref{teo:teo62} part 2. both local cohomologies are trivial.
Therefore $[Z] \in \ker ( r_L ) \cap \ker ( r_U )$.
In other words the restriction of $Z$ to $ {\mathpzc{L}}_n$ and its restriction to ${\mathpzc{U}}_n$ are coboundaries.
Hence, there exist $P, Q \in \E[p^n]$, such that
\begin{equation}\label{eqn:rel64}
\begin{split}
Z_\gamma &= \gamma ( P ) - P  \quad{\rm for \quad every}\quad \gamma \in {\mathpzc{L}}_n; \\  \ Z_\delta &= \delta ( Q ) - Q \quad{\rm for \quad every} \quad\delta \in {\mathpzc{U}}_n.
\end {split}\end{equation}
Observe that  $\rho_n \in {\mathpzc{L}}_n \cap {\mathpzc{U}}_n$. Then
\[
Z_{\rho_n} = \rho_n ( P ) - P = \rho_n ( Q ) - Q.
\]
Thus $P-Q \in \ker ( \rho_n - I )$.
Modulo $p$, the eigenvalues of $\rho_n$ coincide with $\lambda_1$ and $\lambda_2$. Therefore $\rho_n - I$ is an isomorphism.
In particular $\ker ( \rho_n - I ) = 0$, which implies $P = Q$.
Then $[Z]$ is $0$ over the group generated by ${\mathpzc{L}}_n$ and ${\mathpzc{U}}_n$.
As $G_n$ is generated by $\rho_n$ and $H_n$,  Proposition~\ref{nuovo1} implies that $G_n$ is generated by ${\mathpzc{L}}_n$ and ${\mathpzc{U}}_n$.
Thus $[Z]$ is $0$ over $G_n$.

Part 2.
Consider the restrictions
\begin{equation*}
\begin{split}
{r_{SL}} &: H^1 ( G_n, \E[p^n] ) {\rightarrow} H^1 ( \langle \rho_n, {\mathpzc{sL}}_n\rangle, \E[p^n] ),\\ {r_U} &: H^1 ( G_n, \E[p^n] ) {\rightarrow} H^1 ( {\mathpzc{U}}_n, \E[p^n] ).
\end{split}
\end{equation*}
Let $Z$ be a cocycle from $G_n$ to $\E[p^n]$, such that its class $[Z] \in H^1_{{\rm loc}} ( G_n, \E[p^n] )$.
Then $r_{SL}( [Z] ) \in H^1_{{\rm loc}} ( \langle \rho_n, {\mathpzc{sL}}_n\rangle, \E[p^n] ) $ and $r_U([Z]) \in H^1_{{\rm loc}} ( {\mathpzc{U}}_n, \E[p^n] )$, which are trivial  by Proposition~\ref{teo:teo62}.
It follows $[Z] \in \ker ( r_{SL} ) \cap \ker ( r_U )$.
Hence, there exist $P, Q \in \E[p^n]$, such that
\begin{equation}\label{eqn:relnuovissima}
\begin{split}
Z_\gamma &= \gamma ( P ) - P  \quad{\rm for \quad every}\quad \gamma \in \langle \rho_n, {\mathpzc{sL}}_n\rangle; \\  \ Z_\delta &= \delta ( Q ) - Q \quad{\rm for \quad every} \quad\delta \in {\mathpzc{U}}_n.
\end {split}\end{equation}
Recall that  $\rho_n \in \langle \rho_n, {\mathpzc{sL}}_n\rangle \cap {\mathpzc{U}}_n$. Then
\begin{equation}\label{eqn:rel65}
Z_{\rho_n}  = \rho_n ( P ) - P = \rho_n ( Q ) - Q
\end{equation}
and  $P - Q \in \ker ( \rho_n - I )$.
Observe that $\rho_n$ has the same order of $\rho$ and $\lambda_2 = 1$. Since the order of $\rho$ divides $p-1$, and $\lambda_{2,n}\equiv \lambda_2$ $(\textrm{mod } p)$,
then $\lambda_{2,n} = 1$ too. We have $
\rho_n - I =
\left(
\begin{array}{cc}
\lambda_{1, n} - 1 & 0 \\
0 & 0 \\
\end{array}
\right),
$
with $\lambda_{1, n} \not \equiv 1 \mod ( p )$, because $\rho$ has order at least $3$ and $\lambda_2 = 1$.
We deduce   $P - Q = ( 0, b )$ for a certain $b \in \Z / p^n \Z$.
In addition, $\langle \rho_n, {\mathpzc{sL}}_n\rangle$ is generated by $\rho_n$ and $\tau = {1\,\,\, 0 \choose p^j\,\,\, 1} $.
By (\ref{eqn:relnuovissima}), we know $Z_{\tau} = \tau ( P ) - P$.
But
$\tau- I =
\left(
\begin{array}{cc}
0 & 0 \\
p^j & 0 \\
\end{array}
\right).
$
Then
$
\tau ( P ) - P = \tau ( P - ( 0, b ) ) - ( P - ( 0, b ) ) = \tau ( Q ) - Q. $
Thus \begin{equation}\label{eqn:rel66}Z_{\tau}  = \tau ( Q ) - Q.\end{equation}

The  group $G_n$ is generated by $\rho_n$ and $H_n$. In view of Proposition \ref{nuovo1}, the group $G_n$ is generated by ${\mathpzc{sL}}_n=\langle\tau\rangle $ and ${\mathpzc{U}}_n$. By (\ref{eqn:rel64}), (\ref{eqn:rel65}) and (\ref{eqn:rel66})
 we see that $Z$ is a coboundary and so  $[Z]=0$.
\end{proof}

\newpage

\thispagestyle{empty}

Laura Paladino\par\smallskip
Dipartimento di Matematica \par
Universit\`{a} della Calabria\par\smallskip
via Ponte Pietro Bucci, cubo 31b\par
IT-87036 Arcavacata di Rende (CS)\par
e-mail address: paladino@mat.unical.it

\vskip 1cm

Gabriele Ranieri\par\smallskip
Collegio Puteano \par
Scuola Normale Superiore \par
Piazza dei Cavalieri 3,\par\smallskip
IT-56100 Pisa,\par
Italy\par
e-mail address: gabriele.ranieri@sns.it

\vskip 1cm

Evelina Viada\par\smallskip
Departement Mathematik\par
Universit\"{a}t Basel\par\smallskip
Rheinsprung, 21\par
CH-4051 Basel\par
e-mail address: evelina.viada@unibas.ch

\end{document}